\documentclass{article}
\usepackage{latexsym}
\usepackage{amsfonts}
\usepackage{amssymb}
\usepackage{amsmath}
\usepackage{amstext}
\usepackage{amsthm}
\usepackage{mathrsfs}

\newtheorem{theorem}{Theorem}
\newtheorem{proposition}[theorem]{Proposition}

\def\proofn{\smallskip\noindent{\mbox{\it Proof of Theorem 1.~}}}
\def\eop{\hfill$\square$}

\begin{document}
\title{On expected and von Neumann-Morgenstern utility functions}
\author{\textbf{Marina Pireddu}\thanks{Many thanks to Michele Gori for suggesting me this interesting problem and for his helpful comments and suggestions.} \\
Dipartimento di Matematica per le Decisioni, \\
Universit\`{a} degli Studi di Firenze \\
e-mail: marina.pireddu@unifi.it}
\maketitle

\begin{abstract}
In this note we analyze the relationship between the properties of von Neumann-Morgenstern utility functions and expected utility functions. 
More precisely, we investigate which of 
the regularity and concavity assumptions usually imposed on the latter transfer to the former and vice versa. In particular we obtain that, in order for the expected utility functions to fulfill such classical properties, it is enough to assume them for the von Neumann-Morgenstern utility functions.
\end{abstract}

\noindent
Let $G=C(S+1),$ where $C\ge 1$ is the number of commodities and $S\ge 1$ is the number of possible states of the world tomorrow. Household's preferences are represented by a utility function
$U:\mathbb{R}_{++}^{G}\to\mathbb{R}$. According to \cite{ali-00} that utility function is named ``expected utility function'' when it has the following form
\begin{equation}\label{eu}
U(x)=\sum_{s=1}^S a_s u(x_0,x_s),
\end{equation}
where $a_s\in(0,1),\,\underset{s=1}{\overset{S}{\sum}}a_s=1$ and $u:\mathbb{R}_{++}^{2C}\to\mathbb{R}$ is the so-called ``von Neumann-Morgenstern utility function''. See also \cite{mas-95}, Chapter 6, for a discussion on the terminology.\\ 
Classical assumptions on utility functions are as follows\footnote{For every positive integer $M,$ given
$v=(v_1,\dots,v_M),\,w=(w_1,\dots,w_M)\in \mathbb{R}^M$, we write
$v\gg w \mbox{  if  } v_i>w_i,$ for each $i\in\{1,\dots,M\}.$}
\begin{align}
&U\in C^{2}(\mathbb{R}_{++}^{G});\label{u1}\\
& \mbox{for every }x\in\mathbb{R}_{++}^{G},\; DU(x)>>0\,;\label{u2}\\
& \mbox{for every }v\in\mathbb{R}^{G}\setminus \{0\}\mbox{ and }x\in\mathbb{R}_{++}^{G},\, 
v \,D^{2}U(x)\,v<0\,;\label{u3}\\
& \mbox{for every }\underline{x}\in\mathbb{R}_{++}^{G},\;
\left\{x\in\mathbb{R}_{++}^{G}:U(x)\geq U(\underline{x})\right\}\mbox{ is closed in the topology of }\mathbb{R}^{G},\label{u4}
\end{align}
with $x=(x_0,x_1,\dots,x_S)\in \mathbb{R}_{++}^{G}$ and $x_i\in\mathbb{R}_{++}^{C},$ for $i\in\{0,1,\dots,S\}\,.$\\
Our aim is that of understanding what are the assumptions to impose on $u$ in order for $U$ in \eqref{eu} to satisfy \eqref{u1}-\eqref{u4}. In particular, we will prove that if $u$ satisfies the analogue of \eqref{u1}-\eqref{u4} on $\mathbb{R}_{++}^{2C},$ i.e.,
\begin{align}
&u\in C^{2}(\mathbb{R}_{++}^{2C});\label{u1m}\\
& \mbox{for every }x\in\mathbb{R}_{++}^{2C},\; Du(x)>>0\,;\label{u2m}\\
& \mbox{for every }w\in\mathbb{R}^{2C}\setminus \{0\}\mbox{ and }x\in\mathbb{R}_{++}^{2C},\, 
w \,D^{2}u(x)\,w<0\,;\label{u3m}\\
& \mbox{for every }\underline{x}\in\mathbb{R}_{++}^{2C},\;
\left\{x\in\mathbb{R}_{++}^{2C}:u(x)\geq u(\underline{x})\right\}\mbox{ is closed in the topology of }\mathbb{R}^{2C},\label{u4m}
\end{align}
then $U$ in \eqref{eu} fulfills \eqref{u1}-\eqref{u4}. Actually, for sake of completeness, we investigate the converse implication, too, obtaining the next result:
\begin{theorem}\label{th}
Let us assume that $u:\mathbb R^{2C}_{++}\to\mathbb R$ is lower unbounded and $U:\mathbb R^{G}_{++}\to\mathbb R$ is as in \eqref{eu}. Then $U$
satisfies \eqref{u1}-\eqref{u4} if and only if $u$ fulfills \eqref{u1m}-\eqref{u4m}. 
\end{theorem}

\noindent
Along the course of the proof of Theorem \ref{th}, we will use the following result:
\begin{proposition}\label{pr}
Let $F:\mathbb R^M_{++}\to\mathbb R$ be continuous and lower unbounded. Then 
$$F(x^{[n]})\to -\infty,\,\,\forall \{x^{[n]}:n\in\mathbb N\}\subseteq \mathbb R^M_{++},\,\, x^{[n]}\to\partial\mathbb R^M_{++}$$
if and only if
$$\mbox{for every }\underline{x}\in\mathbb{R}_{++}^{M},\;
\left\{x\in\mathbb{R}_{++}^{M}:F(x)\geq F(\underline{x})\right\}\mbox{ is closed in the topology of }\mathbb{R}^{M}.$$
\end{proposition}
\begin{proof}
$\Rightarrow:$ Assume by contradiction that $F(x^{[n]})\to -\infty,\,\forall \{x^{[n]}:n\in\mathbb N\}\subseteq \mathbb R^{M}_{++},$ $x^{[n]}\to\partial\mathbb R^{M}_{++}$ but there exists $\underline{x}\in\mathbb{R}_{++}^{M}$ such that
$S(\underline{x})=\left\{x\in\mathbb{R}_{++}^{M}:F(x)\geq F(\underline{x})\right\}$ is not closed in the topology of $\mathbb{R}^{M}.$ Then there exists a sequence $(\widehat x^{[n]})_n$ in $S(\underline{x})$ converging to a certain $x^*\notin S(\underline{x}).$ This means that $x^*\in \partial\mathbb R^{M}_{++}$ or $F(x^*)< F(\underline{x}).$ The latter possibility is however prevented by the continuity of $F.$ Then it has to be $x^*\in\partial\mathbb R^{M}_{++}.$ As a consequence $\widehat x^{[n]}\to x^*\in \partial\mathbb R^{M}_{++}$ and thus $F(\widehat x^{[n]})\to -\infty,$ but this is impossible as $F(\widehat x^{[n]})\geq F(\underline{x}),$ for any $n\in\mathbb N.$\\
$\Leftarrow:$ Assume by contradiction that for every $\underline{x}\in\mathbb{R}_{++}^{M},\;S(\underline{x})$ is closed in the topology of $\mathbb{R}^{M}$ but there exists $\{\widehat x^{[n]}:n\in\mathbb N\}\subseteq \mathbb R^{M}_{++},\,\widehat x^{[n]}\to\partial\mathbb R^{M}_{++}$ with $F(\widehat x^{[n]})\not\to -\infty.$ This means that 
\begin{equation}\label{seq}
\exists K>0: \forall\bar n\in\mathbb N,\,\exists n\ge\bar n \mbox{ with } F(\widehat x^{[n]})>-K.
\end{equation} 
As $F$ is lower unbounded, there exists ${x}^*\in\mathbb{R}_{++}^{M}$ such that $F({x}^*)=-K.$ But then $S({x}^*)$ is not closed in the topology of $\mathbb{R}^{M}.$ Indeed, for a subsequence of $(\widehat x^{[n]})_n$ as in \eqref{seq}, that we still denote by $(\widehat x^{[n]})_n,$ it holds that $F(\widehat x^{[n]})>-K=F({x}^*),$ and thus $(\widehat x^{[n]})_n$ is in $S({x}^*).$ However, $\widehat x^{[n]}\to\partial\mathbb R^{M}_{++}$ and so its limit point does not belong to $S({x}^*).$ The contradiction is found.
\end{proof}

\smallskip

\proofn
About \eqref{u1m}$\Rightarrow$\eqref{u1}, it is immediate to see that if $u$ is $C^2$ on $\mathbb R^{2C}_{++},$ then $U$ has the same property on $\mathbb R^{G}_{++}.$\\ 
Vice versa, as 
\begin{equation}\label{comp}
u(x_0,x_1)=a_1 u(x_0,x_1)+\dots+a_{S-1} u(x_0,x_1)+\Big(1-\sum_{s=1}^{S-1}a_s\Big) u(x_0,x_1)=U(x_0,x_1,\dots,x_1),
\end{equation}
if $U\in C^2(\mathbb R^{G}_{++}),$ then $u\in C^2(\mathbb R^{2C}_{++})$ and thus \eqref{u1}$\Rightarrow$\eqref{u1m}.

\smallskip

\noindent
It is straightforward to see that if $u$ fulfills \eqref{u2m} then $U$ in \eqref{eu} satifies \eqref{u2}, as
\begin{equation}\label{du}
DU(x_0,x_1,\dots,x_S)=\left(\underset{s=1}{\overset{S}{\sum}}a_s D_{x_0}u(x_0,x_s),\,a_1\, D_{x_1}u(x_0,x_1),\dots, a_S\, D_{x_S}u(x_0,x_S)\right),
\end{equation}
where $D_{x_i}u(x_0,x_s)$ denotes the partial Jacobian of $u(x_0,x_s)$ with respect to $x_i,$ for $i\in\{0,s\}$ and $s\in\{1,\dots,S\}.$\\
The implication \eqref{u2}$\Rightarrow$\eqref{u2m} follows again by \eqref{comp}, as 
\begin{equation}\label{d1}
Du(x,y)=DU(x,y,\dots,y)=(D_{x_0}U(x,y,\dots,y),D_{x_1}U(x,y,\dots,y)+\dots+D_{x_S}U(x,y,\dots,y)),
\end{equation}
where, for $U(x_0,x_1,\dots,x_S),$ we have set $D_{x_s}U(x,y,\dots,y)=D_{x_s}U(x_0,x_1,\dots,x_S)|_{(x,y,\dots,y)},$ with $s\in\{0,\dots,S\}\,.$

\smallskip

\noindent
Let us turn to \eqref{u3m}$\Rightarrow$\eqref{u3}. The computations of $D^2 u(x_0,x_s),$ with $s\in\{1,\dots,S\},$ and $D^2 U(x_0,x_1,\dots,x_S)$ are as follows\footnote{Notice that throughout the note, since \eqref{u1} and \eqref{u1m} hold true, we use Schwarz Lemma when computing the partial Hessians of $u$ and $U.$}:
\begin{equation}\label{d2us}
D^2 u(x_0,x_s)=\left(
\begin{array}
[c]{ll}
D^2_{x_0,x_0}u(x_0,x_s)& \,\, D^2_{x_0,x_s}u(x_0,x_s)\\
\\
D^2_{x_0,x_s}u(x_0,x_s) & \,\,D^2_{x_s,x_s}u(x_0,x_s)
\end{array}
\right)
\end{equation}
and $D^2 U(x_0,x_1,\dots,x_S)=$
\begin{equation}\label{d2u}
\left(
\begin{array}
[c]{ccccc}
\underset{s=1}{\overset{S}{\sum}} a_s\, D^2_{x_0,x_0}u(x_0,x_s) & a_1\, D^2_{x_0,x_1}u(x_0,x_1) &  a_2\, D^2_{x_0,x_2}u(x_0,x_2) & \dots & a_S\, D^2_{x_0,x_S}u(x_0,x_S)\\
\\
a_1\, D^2_{x_0,x_1}u(x_0,x_1) & a_1\, D^2_{x_1,x_1}u(x_0,x_1) & 0 & \ldots & 0\\
\\
a_2\, D^2_{x_0,x_2}u(x_0,x_2) &  0 & a_2\, D^2_{x_2,x_2}u(x_0,x_2) &  & 0\\
\\
\vdots & \vdots & & \ddots & \\
\\
a_S\, D^2_{x_0,x_S}u(x_0,x_S) & 0 & \dots & 0 & a_S\, D^2_{x_S,x_S}u(x_0,x_S)
\end{array}
\right)
\end{equation}
where $D^2_{x_i,x_j}u(x_0,x_s)$ denotes the partial Hessian of $u(x_0,x_s)$ with respect to $x_i$ and $x_j,$ for $i,j\in\{0,s\}$ and $s\in\{1,\dots,S\}.$\\
Since $u:\mathbb R^{2C}_{++}\to\mathbb R$ satisfies \eqref{u3m}, for any $w=(w_0,w_s)\in\mathbb{R}^{2C}\setminus\{0\}$ and $(x,y)\in\mathbb{R}_{++}^{2C},\, 
w \,D^{2}u(x,y)\,w<0,$ i.e., using \eqref{d2us},
\begin{equation}\label{w1}
w_0D^2_{x_0,x_0}u(x,y)w_0+2w_0D^2_{x_0,x_s}u(x,y)w_s+w_sD^2_{x_s,x_s}u(x,y)w_s<0,
\end{equation}
where, for $u(x_0,x_s),$ we have set $D^2_{x_i,x_j}u(x,y)=D^2_{x_i,x_j}u(x_0,x_s)|_{(x,y)},$ for $i,j\in\{0,s\}$ and $s\in\{1,\dots,S\}\,.$
We want to prove that $U:\mathbb R^{G}_{++}\to\mathbb R$ fulfills \eqref{u3}, i.e., for any $v=(v_0,v_1,\dots,v_S)\in\mathbb{R}^{G}\setminus \{0\}$ and $(\overline x_0,\overline x_1,\dots,\overline x_S)\in\mathbb{R}_{++}^{G},$
\begin{equation}\label{v1}
\begin{array}{ll}
&\left(v_0\left(\underset{s=1}{\overset{S}{\sum}} a_s\, D^2_{x_0,x_0}u(\overline x_0,\overline x_s)\right)+v_1 a_1 D^2_{x_0,x_1}u(\overline x_0,\overline x_1)+\dots+v_S a_S D^2_{x_0,x_S}u(\overline x_0,\overline x_S)\right)v_0+\\
\vspace{-2mm}\\
&\left(v_0 a_1 D^2_{x_0,x_1}u(\overline x_0,\overline x_1)+v_1 a_1 D^2_{x_1,x_1}u(\overline x_0,\overline x_1)\right)v_1+\\
&\qquad\qquad\qquad\qquad\qquad\quad\vdots\\
&\left(v_0 a_S D^2_{x_0,x_S}u(\overline x_0,\overline x_S)+v_S a_S D^2_{x_S,x_S}u(\overline x_0,\overline x_S)\right)v_S<0,
\end{array}
\end{equation}
where we have used \eqref{d2u}. Rewriting \eqref{v1} as 
\begin{equation}\label{v2}
\begin{array}{ll}
&a_1\left(v_0D^2_{x_0,x_0}u(\overline x_0,\overline x_1)v_0+2v_0D^2_{x_0,x_1}u(\overline x_0,\overline x_1)v_1+v_1D^2_{x_1,x_1}u(\overline x_0,\overline x_1)v_1\right)+\\
&\qquad\qquad\qquad\qquad\qquad\qquad\qquad\qquad\vdots\\
&a_S\left(v_0D^2_{x_0,x_0}u(\overline x_0,\overline x_S)v_0+2v_0 D^2_{x_0,x_S}u(\overline x_0,\overline x_S)v_S+v_SD^2_{x_S,x_S}u(\overline x_0,\overline x_S)v_S\right)<0,
\end{array}
\end{equation}
it is immediate to see that the desired property follows by \eqref{w1}, choosing $w=(v_0,v_s)$ and $(x,y)=(\overline x_0,\overline x_s),$ with $s\in\{1,\dots,S\}\,.$ \footnote{Actually, in order to employ \eqref{u3m}, we should know that, for $s\in\{1,\dots,S\},$ $(v_0,v_s)$ is a nonnull vector, since otherwise the left hand-side in \eqref{w1} would be equal to $0.$ However, since $v=(v_0,v_1,\dots,v_S)\in\mathbb{R}^{G}\setminus \{0\},$ at least one in $\{(v_0,v_1),\dots,(v_0,v_S)\}$ is nonnull and this is sufficient to conclude that \eqref{v2} holds true.}\\
Vice versa, let us assume that  $U:\mathbb R^{G}_{++}\to\mathbb R$ in \eqref{eu} fulfills \eqref{u3} and let us show that $u:\mathbb R^{2C}_{++}\to\mathbb R$ satisfies \eqref{u3m}. Using \eqref{comp} and \eqref{d1}, we get that 
\begin{equation}\label{d2}
\begin{array}{l}
D^2 u(x,y)=D^2 U(x,y,\dots,y)=\\
\\
\left(
\begin{array}
[c]{ll}
D^2_{x_0,x_0}U(x,y,\dots,y) & \,\, \underset{s\in\{1,\dots,S\}}{\sum}D^2_{x_0,x_s}U(x,y,\dots,y)\\
\\
\underset{s\in\{1,\dots,S\}}{{\sum}}D^2_{x_0,x_s}U(x,y,\dots,y) & \,\,
\underset{i,j\in\{1,\dots,S\}}{{\sum}}D^2_{x_i,x_j}U(x,y,\dots,y)
\end{array}
\right)
\end{array}
\end{equation}
where, for $U(x_0,x_1,\dots,x_S),$ we have set $D^2_{x_i,x_j}U(x,y,\dots,y)=D^2_{x_i,x_j}U(x_0,x_1,\dots,x_S)|_{(x,y,\dots,y)},$ with $i,j\in\{0,\dots,S\}\,.$\\
Since $U:\mathbb R^{G}_{++}\to\mathbb R$ in \eqref{eu} fulfills \eqref{u3}, for any $v=(v_0,v_1,\dots,v_S)\in\mathbb{R}^{G}\setminus \{0\}$ and $(\overline x_0,\overline x_1,\dots,\overline x_S)\in\mathbb{R}_{++}^{G},$ 
$v \,D^{2}U(\overline x_0,\overline x_1,\dots,\overline x_S)\,v<0,$ i.e.,
\begin{equation}\label{v}
\sum_{i,j\in\{0,\dots,S\}}v_i D^2_{x_i,x_j}U(\overline x_0,\overline x_1,\dots,\overline x_S)v_j<0.
\end{equation}
Using \eqref{d2}, we rewrite \eqref{u3m} as follows:
for any $w=(z,t)\in\mathbb{R}^{2C}\setminus \{0\}$ and $(x,y)\in\mathbb{R}_{++}^{2C},$ 
\begin{equation}\label{w}
zD^2_{x_0,x_0}U(x,y,\dots,y)z+2\,z\sum_{s=1}^S D^2_{x_0,x_s}U(x,y,\dots,y)\,t+t\sum_{i,j\in\{1,\dots,S\}}D^2_{x_i,x_j}U(x,y,\dots,y)\,t<0.
\end{equation}
But that comes by \eqref{v}, choosing $v=(z,t,\dots,t)$ and $(\overline x_0,\overline x_1,\dots,\overline x_S)=(x,y,\dots,y).$
 
\smallskip

\noindent
Let us now assume that the continuous and lower unbounded function $u:\mathbb R^{2C}_{++}\to\mathbb R$ satisfies \eqref{u4m} and we prove that $U$ in \eqref{eu} fulfills $\eqref{u4}.$ Applying Proposition \ref{pr} to $u,$ we find that $u(x_0^{[n]},x_s^{[n]})\to -\infty,$ for any sequence $\{(x_0^{[n]},x_s^{[n]}):n\in\mathbb N\}\subseteq\mathbb R^{2C}_{++},\,(x_0^{[n]},x_s^{[n]})\to\partial\mathbb R^{2C}_{++},$ for $s\in\{1,\dots,S\}.$  We are going to check that $U$ in \eqref{eu} satisfies the following property
\begin{equation}\label{u4e}
U(x_0^{[n]},x_1^{[n]},\dots,x_S^{[n]})\to -\infty,\,\forall \{(x_0^{[n]},x_1^{[n]},\dots,x_S^{[n]}):n\in\mathbb N\}\subseteq\mathbb R^{G}_{++},\,(x_0^{[n]},x_1^{[n]},\dots,x_S^{[n]})\to\partial\mathbb R^{G}_{++},
\end{equation}
so that, by Proposition \ref{pr}, we can conclude that $U$ fulfills $\eqref{u4},$ as desired. Notice that if $u$ is continuous and lower unbounded, then $U$ displays the same properties and thus it is in fact possible to apply Proposition \ref{pr} to that function.
Let $(x_0^{[n]},x_1^{[n]},\dots,x_S^{[n]})_{{}_n}$ be a sequence in $\mathbb R^{G}_{++}$ tending to $\partial\mathbb R^{G}_{++}.$ Then $x_{\bar s}^{[n]}\to\partial\mathbb R^{C}_{++},$ for at least one $\bar s\in\{0,1,\dots,S+1\}.$ If $\bar s\in\{1,\dots,S\},$ then $u(x_0^{[n]},x_{\bar s}^{[n]})\to -\infty.$ If instead $\bar s=0,$ then $u(x_0^{[n]},x_{s}^{[n]})\to -\infty,$ for every $s\in\{1,\dots,S\}.$ In both cases,   
$U(x_0^{[n]},x_1^{[n]},\dots,x_S^{[n]})\to-\infty$ because it is the sum of terms tending to $-\infty$ plus, by the continuity of $u,$ bounded quantities (if any). Condition \eqref{u4e} is thus checked.\\ 
The implication \eqref{u4}$\Rightarrow$\eqref{u4m} follows by \eqref{comp}, applying again Proposition \ref{pr} to $U$ and $u.$
\eop

\bigskip

We end this note investigating what happens if we replace \eqref{u3} and \eqref{u3m} with the weaker conditions
\begin{equation}\label{u3w}
\mbox{for every }v\in\mathbb{R}^{G}\setminus \{0\}\mbox{ and }x\in\mathbb{R}_{++}^{G}, \;DU(x)\, v=0\, \mbox{ implies }\,v \,D^{2}U(x)\,v<0
\end{equation}
and 
\begin{equation}\label{u3mw}
\mbox{ \,\,\,\,for every }w\in\mathbb{R}^{2C}\setminus \{0\}\mbox{ and }x\in\mathbb{R}_{++}^{2C},\;Du(x)\, w=0\, \mbox{ implies }\, w \,D^{2}u(x)\,w<0,
\end{equation}
respectively.\\
We notice that an argument similar to the one used in the proof of Theorem \ref{th} to check that \eqref{u3}$\Rightarrow$\eqref{u3m} also shows that if $U:\mathbb R^{G}_{++}\to\mathbb R$ in \eqref{eu} satisfies \eqref{u3w} then $u:\mathbb R^{2C}_{++}\to\mathbb R$ fulfills \eqref{u3mw}.\\
Indeed let us assume that, whenever $DU(\overline x_0,\overline x_1,\dots,\overline x_S)v=0,$ then \eqref{v} holds, where $v=(v_0,v_1,\dots,v_S)\in\mathbb{R}^{G}\setminus \{0\}$ and $(\overline x_0,\overline x_1,\dots,\overline x_S)\in\mathbb R^{G}_{++}.$ We have to show that if $Du(x,y)w=0$ then \eqref{w} holds true, where $w=(z,t)\in\mathbb{R}^{2C}\setminus \{0\}$ and $(x,y)\in\mathbb{R}^{2C}_{++}.$ But this follows by \eqref{v} choosing $v=(z,t,\dots,t)$ and 
$(\overline x_0,\overline x_1,\dots,\overline x_S)=(x,y,\dots,y)$ since, by \eqref{d1},
$$
\begin{array}{ll}
Du(x,y)w=DU(x,y,\dots,y)w=\\
D_{x_0}U(x,y,\dots,y)z+\left(D_{x_1}U(x,y,\dots,y)+\dots+D_{x_S}U(x,y,\dots,y)\right)t=
DU(\overline x_0,\overline x_1,\dots,\overline x_S)v,
\end{array}
$$
where, for $U(x_0,x_1,\dots,x_S),$ we have set $D_{x_s}U(x,y,\dots,y)=D_{x_s}U(x_0,x_1,\dots,x_S)|_{(x,y,\dots,y)},$ with $s\in\{0,\dots,S\}\,.$\\
Instead the vice versa, i.e., that if $u:\mathbb R^{2C}_{++}\to\mathbb R$ satisfies \eqref{u3mw}, then $U:\mathbb R^{G}_{++}\to\mathbb R$ in \eqref{eu} fulfills \eqref{u3w}, does not seem to hold in general. 
We try to explain what is the point. By \eqref{du}, for $v=(v_0,v_1,\dots,v_S)\in\mathbb{R}^{G}\setminus \{0\}$ and $(\overline x_0,\overline x_1,\dots,\overline x_S)\in\mathbb R^{G}_{++},$ $DU(\overline x_0,\overline x_1,\dots,\overline x_S)v=0$ can be written as
\begin{equation}\label{eq}
\begin{array}{l}
DU(\overline x_0,\overline x_1,\dots,\overline x_S)v=\biggl(\underset{s=1}{\overset{S}{\sum}}a_s D_{x_0}u(\overline x_0,\overline x_s),
a_1 D_{x_1}u(\overline x_0,\overline x_1),\dots, a_S D_{x_S}u(\overline x_0,\overline x_S)\biggr)v=\\
a_1\left(D_{x_0}u(\overline x_0,\overline x_1)v_0+D_{x_1}u(\overline x_0,\overline x_1)v_1\right)+\dots+a_S\left(D_{x_0}u(\overline x_0,\overline x_S)v_0+D_{x_S}u(\overline x_0,\overline x_S)v_S\right)=0,
\end{array}
\end{equation}
where, for $u(x_0,x_s),$ we have set $D_{x_i}u(\overline x_0,\overline x_s)=D_{x_i}u(x_0,x_s)|_{(\overline x_0,\overline x_s)},$
with $i\in\{0,s\}$ and $s\in\{1,\dots,S\}.$ \\
In order to use \eqref{u3mw} with $w=(v_0,v_s)$ and $x=(\overline x_0,\overline x_s),$ for $s\in\{1,\dots,S\},$ to obtain\footnote{Actually, in order to employ \eqref{u3mw}, we should know that $(v_0,v_s),$ for $s\in\{1,\dots,S\},$ are nonnull vectors, since otherwise the left hand-side in \eqref{bar1} would be equal to $0.$ However, since $v=(v_0,v_1,\dots,v_S)\in\mathbb{R}^{G}\setminus \{0\},$ at least one in $\{(v_0,v_1),\dots,(v_0,v_S)\}$ is nonnull and this is sufficient to conclude that \eqref{dd2} holds true.} 
\begin{equation}\label{bar1}
v_0D^2_{x_0,x_0}u(\overline x_0,\overline x_s)v_0+2v_0D^2_{x_0,x_s}u(\overline x_0,\overline x_s)v_s +v_sD^2_{x_s,x_s}U(\overline x_0,\overline x_s)v_s <0,\quad s\in\{1,\dots,S\},
\end{equation}
so that, recalling \eqref{d2u},
\begin{equation}\label{dd2}
\begin{array}{ll}
\!\!\!\!\!\!\!v D^2U(\overline x_0,\overline x_1,\dots,\overline x_S)v=
a_1\left(v_0D^2_{x_0,x_0}u(\overline x_0,\overline x_1)v_0+2v_0D^2_{x_0,x_1}u(\overline x_0,\overline x_1)v_1+v_1D^2_{x_1,x_1}u(\overline x_0,\overline x_1)v_1\right)+\\
\qquad\qquad\qquad\qquad\qquad\qquad\qquad\qquad\qquad\qquad\qquad\qquad\,\,\,\,\vdots\\
\qquad\qquad\qquad\qquad\qquad a_S\left(v_0D^2_{x_0,x_0}u(\overline x_0,\overline x_S)v_0+2v_0 D^2_{x_0,x_S}u(\overline x_0,\overline x_S)v_S+v_SD^2_{x_S,x_S}u(\overline x_0,\overline x_S)v_S\right)<0,
\end{array}
\end{equation}
we should know that $D_{x_0}u(\overline x_0,\overline x_s)v_0+D_{x_s}u(\overline x_0,\overline x_s)v_s=0$, for $s\in\{1,\dots,S\}.$ This is however only a sufficient condition in order for \eqref{eq} to hold, but not a necessary one. For such reason we have no arguments to conclude that \eqref{dd2} holds true every time that \eqref{eq} is satisfied.

\end{document}